\documentclass[10pt]{amsart}

\usepackage[
naturalnames,
bookmarks=false
]{hyperref}

\usepackage{amsfonts,amssymb}
\usepackage{amsmath}
\usepackage{mathrsfs}
\usepackage{latexsym}
\usepackage{amscd}
\begin{document}


\newtheorem{thm}{Theorem}[section]
\newtheorem{cor}[thm]{Corollary}
\newtheorem{lem}[thm]{Lemma}
\newtheorem{prop}[thm]{Proposition}
\newtheorem{defn}[thm]{Definition}
\newtheorem{rmk}[thm]{Remark}
\newtheorem{exa}[thm]{Example}

\newtheoremstyle{named}{}{}{\itshape}{}{\bfseries}{.}{.5em}{\thmnote{#3}}
\theoremstyle{named}
\newtheorem*{namedtheorem}{}

\newcommand{\nat}{\mathbb{N}}
\newcommand{\reals}{\mathbb{R}}
\newcommand{\integ}{\mathbb{Z}}
\newcommand{\Lor}[1]{\text{Lor}(#1)}
\newcommand{\Euc}[1]{\text{Euc}(#1)}
\newcommand{\mat}[1]{\text{Mat}(#1,\reals)}
\newcommand{\gr}[2]{\text{Gr}_{#1}(#2)}
\newcommand{\SO}[1]{\text{SO}(#1)}
\newcommand{\OO}[1]{\text{O}(#1)}
\newcommand{\GL}[1]{\mathrm{GL}(#1)}
\newcommand{\GLnot}[1]{\mathrm{GL}_0(#1)}
\newcommand{\gau}[1]{\mathrm{Gau}(#1)}
\newcommand{\gaunot}[1]{\mathrm{Gau}_0(#1)}
\newcommand{\aut}[1]{\mathrm{Aut}(#1)}
\renewcommand{\hom}[1]{\mathrm{Hom}(#1)}
\newcommand{\proj}[1]{\mathbb{RP}^{#1}}
\newcommand{\ot}{\leftarrow}
\newcommand{\id}{\text{id}}
\newcommand{\stab}[1]{\text{Stab}(#1)}
\newcommand{\xto}{\xrightarrow}
\newcommand{\sym}[1]{\mathrm{Sym}(#1)}
\newcommand{\met}[2]{\mathrm{Met}_{#1}(#2)}
\newcommand{\norm}[1]{\lVert#1\rVert}
\newcommand{\abs}[1]{\lvert#1\rvert}
\newcommand{\qu}[2]{\mathrm{Qu}_{#1}(#2)}
\newcommand{\map}[1]{\mathrm{Map}(#1)}
\newcommand{\homeo}[1]{\mathrm{Homeo}(#1)}
\newcommand{\nd}{\mathrm{nondeg}}
\newcommand{\splitt}[1]{\mathrm{Split}(#1)}
\newcommand{\fr}[2]{\mathrm{Fr}_{#1}(#2)}
\newcommand{\realproj}[1]{\mathbb{RP}^{#1}}
\newcommand{\zedtwo}{\integ/2\integ}

\title{Sylvester's law of inertia for quadratic forms on vector bundles} \author{Giacomo Dossena}\thanks{Email address: dossena@sissa.it}

\date{}

\begin{abstract}
This paper presents a generalisation of Sylvester's law of inertia to real non-degenerate quadratic forms on a fixed real vector bundle over a connected locally connected paracompact Hausdorff space. By interpreting the classical inertia as a complete discrete invariant for the natural action of the general linear group on quadratic forms, the simplest generalisation consists in substituting such group with the group of gauge transformations of the bundle. Contrary to the classical law of inertia, here the full action and its restriction to the identity path component typically have different orbits, leading to two invariants: a complete invariant for the full action is given by the isomorphism class of the orthonormal frame bundle associated to a quadratic form, while a complete invariant for the restricted action is the homotopy class of any maximal positive-definite subbundle associated to a quadratic form. The latter invariant is finer than the former, which in turn is finer than inertia. Moreover, the orbit structure thus obtained might be used to shed light on the topology of the space of non-degenerate quadratic forms on a vector bundle.
\end{abstract}

\maketitle

\section{Introduction}
Sylvester's law of inertia arose from the efforts to understand real homogeneous quadratic polynomials in $n$ indeterminates. The simplest such polynomial is the sum of $m\leq n$ squares, and the next simplest is a linear combination of $m\leq n$ squares with coefficients taking values in $\{\pm 1\}$. In the first half of the 19th century it was recognised that by a \emph{linear} change of indeterminates (the only kind of change preserving the degree) it is possible to put any homogeneous quadratic polynomial in this second simplest form. Sylvester further observed that the numbers of positive and negative coefficients in the linear combination of squares do not depend on the change of coordinates used for the simplification, and therefore is an intrinsic property of the polynomial, which he dubbed \emph{inertia}\footnote{Because of his poetic disposition, Sylvester indulged in giving imaginative names to many mathematical objects or properties he studied (e.g. matrix, discriminant, invariant, totient, syzygy). Inertia is no exception: in his own words \cite{Syl52},
\begin{quote}
[\dots] a law to which my view of the physical meaning of quantity of matter inclines me, upon the ground of analogy, to give the name of the Law of Inertia for Quadratic Forms, as expressing the fact of the existence of an invariable number inseparably attached to such forms.
\end{quote}}. In the context of abstract algebra such polynomials are identified with quadratic forms on a real vector space of dimension $n$, and Sylvester's law of inertia admits several equivalent formulations which we review in the next section. One of these interprets the inertia as a complete discrete invariant for the natural action of the general linear group on the space of quadratic forms, and it is this formulation which we generalise to non-degenerate quadratic forms on a vector bundle.

Non-degenerate quadratic forms on a vector bundle $E\to X$ appear quite naturally in mathematics and physics, notably in Riemannian geometry and in general relativity where $E$ is the tangent bundle of a manifold. The general linear group $\GL{V}$ of a vector space $V$ can here be substituted by the group $\aut{E}$ of bundle automorphisms, and we further restrict our attention to the subgroup of vertical automorphisms, that is the gauge group $\gau{E}$. The generalisation consists in finding a complete invariant for the action of the gauge group on the space of non-degenerate quadratic forms on the bundle. Contrary to the vector space case, here the full action and its restriction to the identity path component $\gaunot{E}$ have different orbits, thus leading to two different invariants. The main theorem is the following.

\begin{thm}
A complete invariant for the full action of $\gau{E}$ on the space of non-degenerate quadratic forms on $E$ is given by the isomorphism class of the orthonormal frame bundle associated to a quadratic form, while a complete invariant for the restricted action is the homotopy class of any maximal positive-definite vector subbundle associated to a quadratic form.
\end{thm}

This orbit structure might be used in principle to investigate the topology of the space of non-degenerate quadratic forms on $E$, modulo the complication of dealing with non locally compact groups. The enumeration of its path components remains a difficult task (e.g. see \cite{Kos02} for the case of non-degenerate quadratic forms of inertia $(1,n-1)$ on tangent bundles). However, by the above theorem the number of path components in a given $\gau{E}$-orbit is bounded by the number of path components of $\gau{E}/\gaunot{E}$.

\section{Real non-degenerate quadratic forms on vector spaces}
Let $q\colon V\to \reals$ be a real quadratic (possibly degenerate) form on a finite dimensional real vector space $V$, so $q(v)=b(v,v)$ for each $v\in V$ where $b$ is a real symmetric bilinear form on $V$. If we fix a basis $e$ of $V$ we can write $q(v)=x^tax=\sum_{ij} x_ia_{ij}x_j$ where $x$ is the column vector of coordinates of $v$ with respect to $e$, that is $v=\sum_i x_ie_i$, and $a_{ij}=b(e_i,e_j)$ is the symmetric matrix representing $b$. If we change basis, say $e'=eg$ for some $g\in \GL{n}$ where $g$ acts on the right by $e'_i=\sum_j g_{ji}e_j$, the coordinates of $v$ change to $x'=g^{-1}x$ and the symmetric matrix representing $q$ becomes $a'=g^tag$ so the value $q(v)=x^tax=x'^ta'x'$ does not change, which is as it should be. In this setting a common formulation of Sylvester's law of inertia is the following.

\begin{namedtheorem}[Sylvester's law of inertia (\cite{Syl52}, 1852)]
There is a basis of $V$ such that the symmetric matrix representing $q$ is diagonal, and the numbers of positive and negative entries on the diagonal are independent\footnote{The existence of a diagonalising basis was known before 1852 by works of Cauchy, Jacobi, and Borchardt. However, the observation that the numbers of positive and negative entries do not change is due to Sylvester.} of the chosen diagonalising basis.
\end{namedtheorem}

The pair $(n_+,n_-)$ of non-negative integers thus associated to a given quadratic form $q$ is called inertia (or signature, but we prefer the former to honor Sylvester's creative mind). Let us now consider the natural right action of $\GL{V}$ on the set $\qu{}{V}$ of quadratic forms on $V$, as defined by $(fq)(v)=q(fv)$ for each $f\in\GL{V}$. If $q(v)=x^tax$ in a basis $e$ as before, then $(fq)(v)=x^tp^tapx$ where $fe_i=\sum_j p_{ji}e_j$ for $p\in\GL{n}$ uniquely determined by $f$. By interpreting a change of basis as an ``active'' transformation of the elements of $V$, that is an automorphism of $V$, Sylvester's law takes the following equivalent form.

\begin{namedtheorem}[Sylvester's law of inertia (orbit version)]
The inertia is a complete discrete invariant for the $\GL{V}$-action on $\qu{}{V}$. In other words, two quadratic forms on $V$ are in the same $\GL{V}$-orbit if and only if they have the same inertia.
\end{namedtheorem}

In the orbit formulation the quite natural question arises whether the action of $\GL{V}$ can be restricted to a subgroup of $\GL{V}$ without affecting the orbit partition. That such is the case can easily be seen if we fix a basis and identify for a moment quadratic forms and real symmetric matrices\footnote{We might dispense with fixing a basis at the cost of making the argument unnecessarily clumsy.}. Then the action of $\GL{V}$ on quadratic forms becomes the action of $\GL{n}$ on symmetric matrices by congruence and, if we consider the $\GL{n}$-orbit passing through a diagonal matrix $\delta$, we see that we have $g^t\delta g=(hg)^t\delta(hg)$ for each $g\in\GL{n}$, where $h$ is the diagonal matrix with diagonal entries equal to $1$ except one entry which is equal to $-1$. Now, since $\GL{n}$ has only two connected components and since $g$ and $hg$ never lie in the same component, we conclude that the full $\GL{n}$-action shares the orbits with its restriction to the identity component $\GLnot{n}$. We record this fact in the following proposition.

\begin{prop}
The $\GL{V}$-orbits coincide with the $\GLnot{V}$-orbits.
\end{prop}

In the next section we shall see how this coincidence disappears in the vector bundle generalisation, where the two actions typically have different orbit partitions.

There is yet another version of Sylvester's law which we shall need. We omit the proof of equivalence because it is elementary.

\begin{namedtheorem}[Sylvester's law of inertia (splitting version)]
For each quadratic form $q\in\qu{}{V}$ there is a (non-canonical) splitting $V=V_+\oplus V_-\oplus V_0$ such that $q$ restricted to each direct summand is respectively positive definite, negative definite, zero. Even though there are infinitely many such splittings, their dimensions depend only upon $q$. Indeed, $(\dim V_+,\dim V_-)$ is the inertia of $q$.
\end{namedtheorem}

The natural appearance of the connected components of $\GL{V}$ suggests to study the action topologically. Evidently $\GL{V}$ has been tacitly assumed to be endowed with the subspace topology induced by the inclusion in $\hom{V}$, where the latter is endowed with its natural Hausdorff topology making it toplinearly isomorphic to $\reals^{n^2}$. If $V$ and $\qu{}{V}$ are also endowed with their natural Hausdorff topologies, then the right actions $\GL{V}\times V\to V$ and $\GL{V}\times\qu{}{V}\to\qu{}{V}$ turn out to be continuous.

In order to simplify the approach, hereafter we restrict our attention to the case of \emph{non-degenerate} quadratic forms. The main reason is that in the vector bundle setting a form which is degenerate at one fibre might change its inertia in a neighbourhood of the fibre, thus requiring a careful study of singularities. Then, hereafter any inertia $(n_+,n_-)$ will implicitly contain the assumption $n_++n_-=n=\dim V$. We denote by $\qu{n_+,n_-}{V}$ the orbit corresponding to inertia $(n_+,n_-)$. We have the following classical result (we include the proof for the sake of completeness).

\begin{prop}\label{prop:open}
Each $\qu{n_+,n_-}{V}$ is open in $\qu{}{V}$.
\end{prop}
\begin{proof}
Take $q\in\qu{n_+,n_-}{V}$ and consider a decomposition $V=V_+\oplus V_-$ into a positive definite $n_+$-plane $V_+$ and a negative definite $n_-$-plane $V_-$. Now put an arbitrary norm $\norm{\cdot}$ on $V$. This induces the norm $\norm{q}:=\sup_{\norm{v}=1}\abs{q(v)}$ on $\qu{}{V}$. Since the unit sphere in $V_+$ is compact, the continuous map $v\mapsto q(v)$ restricted to this sphere attains a minimum $r_+>0$, i.e. $q(v)\geq r_+$ for each $v$ in the unit sphere in $V_+$. For each $v\in V_+$ we can write $q(v/\norm{v})\geq r_+$, that is $q(v)\geq r_+\norm{v}^2$. Analogously we find a maximum $-r_-<0$ on the unit sphere in $V_-$, therefore for each $w\in V_-$ we get $q(w)\leq -r_-\norm{w}^2$. Now, for any $q'\in\qu{}{V}$ such that $\norm{q-q'}\leq\frac{1}{2}\min(r_+,r_-)$ we have $q'(v)\geq \frac{1}{2}r_+\norm{v}^2$ for each $v\in V_+$, and $q'(w)\leq -\frac{1}{2}r_-\norm{w}^2$ for each $w\in V_-$, hence $q'\in\qu{n_+,n_-}{V}$. This shows that $\qu{n_+,n_-}{V}$ is open in $\qu{}{V}$.
\end{proof}

\section{Real non-degenerate quadratic forms on vector bundles}
We now generalise the above treatment to non-degenerate quadratic forms on vector bundles. Let $\pi\colon E\to X$ (also denoted simply by $E\to X$ or just $E$) be a locally trivial real vector bundle of rank $n$ on a connected locally connected paracompact Hausdorff topological space $X$. A non-degenerate quadratic form on $E$ is a continuous function $q\colon E\to \reals$ such that for each $x\in X$ the restriction $q_x=q|_{E_x}\colon E_x=\pi^{-1}(x)\to \reals$ is a non-degenerate quadratic form on $E_x$.

By Proposition \ref{prop:open} and the connectedness assumption on $X$ the inertia of a non-degenerate quadratic form on $E$ is constant. As already anticipated, this is not true in general if we drop the assumption of non-degeneracy, making the corresponding theory much more complicated. The set $\qu{\nd}{E}$ of non-degenerate quadratic forms on $E$ is thus the disjoint union of subsets $\qu{n_+,n_-}{E}$ of quadratic forms on $E$ with inertia $(n_+,n_-)$, where $(n_+, n_-)$ runs over all ordered pairs of non-negative integers summing to $n$.

The group $\GL{V}$ is now substituted by the group $\gau{E}$ of bundle automorphisms covering the identity, that is all homeomorphisms $\phi\colon E\to E$ such that $\pi\circ\phi=\pi$ and such that they restrict to linear isomorphisms on each fibre.

For topological spaces $Y$ and $Z$ we denote by $\map{Y,Z}$ the set of continuous maps from $Y$ to $Z$ endowed with the compact-open topology. We then consider $\gau{E}$ and $\qu{\nd}{E}$ as endowed with the subspace topologies induced by the inclusions in $\map{E,E}$ and $\map{E,\reals}$ respectively. Since $E$ is locally compact and locally connected\footnote{This follows from the assumption of local connectedness of $X$.}, by a result of Arens \cite{Are46} $\homeo{E}\subset\map{E,E}$ is a topological group and so is the subgroup $\gau{E}\subset\homeo{E}$ with the subspace topology. Moreover, the natural action $\gau{E}\times \qu{\nd}{E}\to\qu{\nd}{E}$ is continuous because it comes from restricting to $\gau{E}\times \qu{\nd}{E}$ the natural assignment $\map{E,E}\times\map{E,\reals}\to\map{E,\reals}$ given by composition of functions. We denote by $\gaunot{E}$ the path component of the identity. By the general theory of topological groups $\gaunot{E}$ is a normal subgroup of $\gau{E}$, hence $\gau{E}/\gaunot{E}$ is a topological group as well.

As suggested by the splitting version of Sylvester's law, we shall need to consider the set of all splittings of $E$ and topologise it. For us a splitting of a vector bundle $E\to X$ is an ordered pair $(E',E'')$ of vector subbundles of $E$ such that $E=E'\oplus E''$, where $\oplus$ denotes the Whitney sum. For a finite dimensional vector space $V$ we denote by $\gr{p}{V}$ the Grassmannian of $p$-planes in $V$, endowed with its standard topology. Given a vector bundle $E\to X$, we denote by $\gr{p}{E}\to X$ the topological fibre bundle whose fibre at $x\in X$ is the Grassmannian $\gr{p}{E_x}$. The space $E^p$ of all rank $p$ subbundles of $E$ can then be defined as the set of continuous sections of $\gr{p}{E}\to X$ endowed with the subspace topology induced by the inclusion in $\map{X,\gr{p}{E}}$. With these definitions the natural action $\gau{E}\times E^p\to E^p$ is continuous, as can be seen by trivialising $E$ over some open covering of $X$. Finally, the set $\splitt{E}$ of all splittings of $E$ is topologised as a subspace of $\sqcup_{j=0}^n E^j\times E^{n-j}$ and the natural action $\gau{E}\times \splitt{E}\to \splitt{E}$ is continuous as well. Notice that every vector bundle has a unique rank $0$ subbundle given by the zero section, denoted by $0$, and a unique rank $n$ subbundle given by $E$ itself, so that we always have the trivial splittings $(E, 0)$ and $(0, E)$.

We now define two equivalence relations on $\splitt{E}$.

\begin{defn}\label{defn:splittings}
Two splittings $E=E'_0\oplus E''_0=E'_1\oplus E''_1$ of a vector bundle $\pi\colon E\to X$ are called isomorphic, and we write $(E'_0,E''_0)\simeq (E'_1,E''_1)$, if there are vector bundle isomorphisms $f'\colon E'_0\to E'_1$ and $f''\colon E''_0\to E''_1$. They are called homotopic, and we write $(E'_0,E''_0)\sim (E'_1,E''_1)$, if there is a splitting $(S',S'')$ of $\pi\times \id\colon E\times[0,1]\to X\times[0,1]$ such that $S' |_{X\times\{i\}}=E'_i$ and $S'' |_{X\times\{i\}}=E''_i$ for $i=0,1$.
\end{defn}

Obviously the ranks of two isomorphic or homotopic splittings are equal.

We can characterise isomorphic or homotopic splittings in terms of the action of $\gau{E}$ or $\gaunot{E}$ accordingly. First we recall the following fundamental lemma.

\begin{lem}\label{lem:homotopy}
Let $E\to X\times[0,1]$ be a vector bundle with $X$ paracompact. There is a vector bundle isomorphism $E\to E |_{X\times\{0\}}\times [0,1]$ which restricts to the identity on $E |_{X\times\{0\}}$.
\end{lem}
\begin{proof}
E.g. see Husemoller \cite{Hus94} Theorem 4.3 and Corollary 4.5 p.29-30, or any textbook on fibre bundles.
\end{proof}

\begin{prop}\label{prop:key}
Two splittings of $E$ are isomorphic if and only if they lie in the same $\gau{E}$-orbit. They are homotopic if and only if they lie in the same $\gaunot{E}$-orbit.
\end{prop}
\begin{proof}
The first statement is easy and does not use the previous lemma. If $(f',f'')$ is an isomorphism between splittings $(E'_0,E''_0)$ and $(E'_1,E''_1)$, then $f=f'\oplus f''$ lies in $\gau{E}$. Conversely, $f\in\gau{E}$ acting on a splitting $(E',E'')$ produces the splitting $(f(E'),f(E''))$ which is isomorphic to $(E',E'')$ by means of $(f',f'')$, where $f'=f |_{E'}$ and $f''=f |_{E''}$. Now we prove the second statement. One direction is immediate: if $f\in\gaunot{E}$ is as in the statement, then choose any path $t\mapsto f_t\in\gaunot{E}$ from $\id_E$ to $f$ and consider the 1-parameter family of splittings $(f_t(E'_0), f_t(E''_0))$ of $E$. This family defines a splitting of $E\times[0,1]\to X\times[0,1]$ which restricts to $(E'_i,E''_i)$ on $X\times\{i\}$, $i=0,1$. Conversely, let $(S',S'')$ be a splitting of $E\times[0,1]\to X\times[0,1]$ which restricts to $(E'_i,E''_i)$ on $X\times\{i\}$, $i=0,1$. Apply lemma \ref{lem:homotopy} to $S'$ and $S''$ separately, obtaining isomorphisms $f'\colon S'\to S' |_{X\times\{1\}}\times [0,1]$ and $f''\colon S''\to S'' |_{X\times\{1\}}\times [0,1]$. Since $S'\oplus S''=E\times[0,1]$ we obtain an isomorphism $f=f'\oplus f''\colon E\times[0,1]\to E\times[0,1]$. Clearly for $(u,t)\in E\times[0,1]$ we have $f(u,t)=(f_t(u),t)$ where the partial map $E\times[0,1]\to E$, $(u,t)\mapsto f_t(u)$ is continuous and such that $f_t\in\gau{E}$ for each $t\in[0,1]$. Moreover by construction we see that $f_0=\id_E$, $f_1(E'_0)=E'_1$ and $f_1(E''_0)=E''_1$. It remains to prove that $t\mapsto f_t$ is continuous. To this end we can view $t\mapsto f_t$ as a map $[0,1]\to \map{E,E}$ associated to $E\times [0,1]\to E$, $(u,t)\mapsto f_t(u)$, so that continuity of $t\mapsto f_t$ is ensured by general properties of the compact-open topology. A more direct proof is as follows. Assume $t\mapsto f_t$ is not continuous at some point, say $t_\star\in[0,1]$. This implies we can find a sequence $t_n\to t_\star$ and a sequence $u_n\in K$ such that $f_{t_n}(u_n)\in E\setminus U$, where $K\subset E$ is compact, $U\subset E$ is open and $f_{t_\star}(K)\subset U$. The continuity of $E\times [0,1]\to E$, $(u,t)\mapsto f_t(u)$ implies $\lim_n f_{t_n}(u_n) = f_{t_\star}(u_\star)$ where $u_\star=\lim_n u_n\in K$. Since $E\setminus U$ is closed, then $f_{t_\star}(u_\star)\in E\setminus U$ and this contradicts the fact that $f_{t_\star}(K)\subset U$.
\end{proof}

We note in passing that the second statement of the above proposition can be viewed as a ``concordance implies isotopy'' result for continuous vector subbundles.

When $X$ is a smooth manifold there is a more explicit construction of the path $t\mapsto f_t$ of automorphisms carrying one splitting to the other. The outline is as follows. The smooth structure on $X$ induces smooth structures on $S'$ and $S''$ (e.g. see \cite{Hir76} Theorem 3.5 p.101) making it possible to choose arbitrary linear connections $\nabla'$ and $\nabla''$ on them. The connection $\nabla=\nabla'\oplus\nabla''$ on $S'\oplus S''=E\times [0,1]$ allows us to parallel transport any $u=u'\oplus u''\in E\times \{0\}$ along the straight path $t\mapsto (\pi(u),t)\in X\times [0,1]$ obtaining a vector $u_t=u'_t\oplus u''_t$. Finally, the map $f_t(u)= u_t$ defines a continuous path of automorphisms of $E$ with the desired properties.

\begin{rmk}\label{rmk:splittings}
Evidently two homotopic splittings are isomorphic. The converse does not hold, as the following example shows.
\end{rmk}

\begin{exa}
Given the trivial rank 2 vector bundle $S^1\times \reals^2\to S^1$ and a natural number $k\in\nat$, consider the line subbundle $l_k$ defined by its corresponding Gauss map $\Phi_k\colon S^1\to\realproj{1}$, $\alpha\mapsto [\cos \frac{k\alpha}{2}: \sin \frac{k\alpha}{2}]$ where $[x:y]$ are homogeneous coordinates in $\realproj{1}$. In other words, $\Phi_k(x)$ is a line through the origin in $\reals^2$ performing $k$ half turns around the origin while $x$ moves around $S^1$. Notice that $k\neq k'$ implies that $\Phi_k$ and $\Phi_{k'}$ correspond to different elements in $\pi_1(\realproj{1})\simeq \pi_1(S^1)\simeq\integ$. Also notice that $l_k$ is trivial if and only if $k$ is even, a nowhere vanishing section being given by $\alpha\mapsto (\cos \frac{k\alpha}{2}, \sin \frac{k\alpha}{2})$. By the well-known fact that on $S^1$ there are only two isomorphism classes of line bundles, one represented by the trivial bundle and the other one represented by the M\"obius bundle, we deduce that $l_k\simeq l_{k'}$ if and only if $k=k'\mod 2$. However, for $k\neq k'$ the line bundles $l_k$ and $l_{k'}$ are not homotopic. Indeed, assume there is a line subbundle $l$ of the trivial rank 2 vector bundle over $S^1\times[0,1]$ such that $l |_{S^1\times\{0\}}=l_k$ and $l |_{S^1\times\{1\}}=l_{k'}$. Then the corresponding Gauss map $\Phi_l\colon S^1\times [0,1]\to \realproj{1}$ determines a homotopy between $\Phi_k$ and $\Phi_{k'}$ as elements of $\pi_1(\realproj{1})$, which implies $k=k'$ as we observed before. Another way to interpret this example is the following: the tautological line bundle $\gamma_1(\reals^3)\to \realproj{2}$ is universal for 1-dimensional manifolds, and the line bundles $l_{2k}$ and $l_{2k+1}$ correspond to the two elements in $\pi_1(\realproj{2})\simeq\zedtwo$. On the other hand, the tautological line bundle $\gamma_1(\reals^2)\to \realproj{1}$ is universal only for 0-dimensional manifolds.
\end{exa}

We now briefly review some basic constructions we shall need later on. We start with a small lemma for quadratic forms on a vector space $V$.

\begin{lem}\label{lem:simul}
Given $q\in\qu{\nd}{V}$ and $q'\in\qu{n,0}{V}$, there is a basis of $V$ that simultaneously diagonalises $q$ and $q'$.
\end{lem}
\begin{proof}
Apply the spectral theorem for symmetric bilinear forms to $q$ relatively to the inner product space $(V,q')$.
\end{proof}

Given $q\in\qu{\nd}{E}$ with associated symmetric bilinear form $$b_q(u,v)=\frac{1}{2}\left(q(u+v)-q(u)-q(v)\right)\;,$$ we denote by $\tilde{q}\colon E\to E^*$ the perfect pairing given by $\tilde{q}(u) = b_q(u,\cdot)$ where $E^*$ is the dual of $E$. Now let us fix a positive definite form $q'\in\qu{n,0}{E}$ and for each $q\in\qu{\nd}{E}$ define $$L_{q'q}=\tilde{q'}^{-1}\circ \tilde{q}\in\gau{E}\;.$$ By construction $b_q(u,v)=b_{q'}(L_{q'q}u,v)$ for any $u,v\in E$, hence the diagonalising basis of Lemma \ref{lem:simul} diagonalises $(L_{q'q})_x\in\GL{E_x}$. In the case $q$ is positive definite the eigenvalues of $(L_{q'q})_x$ are positive and we write $(\sqrt{L}_{q'q})_x$ for its unique positive square root. By the continuity of the operation $M\mapsto \sqrt{M}$ for positive definite matrices $M$ together with the local triviality of $E$ we get $\sqrt{L}_{q'q}\in\gau{E}$. Moreover, since $L_{q'q}$ is $b_{q'}$-symmetric (as follows from $b_q(u,v)=b_{q'}(L_{q'q}u,v)$), by the spectral calculus also $\sqrt{L}_{q'q}$ is $b_{q'}$-symmetric, so we get $b(u,v)=b_{q'}(\sqrt{L}_{q'q}u,\sqrt{L}_{q'q}v)$ for any $u,v\in E$, that is $$q(v)=q'(\sqrt{L}_{q'q}v)\quad\text{for each}\; v\in E\;.$$

Among all non-degenerate quadratic forms, the positive definite ones play a distinctive role due to the following well-known fundamental fact.

\begin{prop}\label{prop:convex}
$\qu{n,0}{E}$ is a convex subset of $\map{E,\reals}$.
\end{prop}
\begin{proof}
For any $q_0,q_1\in\qu{n,0}{E}$ and any real numbers $t_0,t_1>0$ we have $t_0q_0+t_1q_1\in\qu{n,0}{E}$, therefore $[0,1]\ni t\mapsto tq_1+(1-t)q_0$ is a straight path in $\qu{n,0}{E}$ from $q_0$ to $q_1$.
\end{proof}

\begin{defn}
Given $q\in\qu{\nd}{E}$, a $q$-splitting of $E$ is a splitting $E=E^+\oplus E^-$ such that $q |_{E^+}$ is positive definite and $q |_{E^-}$ is negative definite.
\end{defn}

\begin{prop}\label{prop:splittings}
Given $q\in\qu{\nd}{E}$, a choice of $r\in\qu{n,0}{E}$ determines a canonical $q$-splitting $E=E^+_q(r)\oplus E^-_q(r)$ and any $q$-splitting arises in this fashion for some $r\in\qu{n,0}{E}$. The homotopy and isomorphism classes of $(E^+_q(r), E^-_q(r))$ are independent of $r$.
\end{prop}
\begin{proof}
Define $E^+_q(r)_x$ to be the direct sum of all eigenspaces of $(L_{rq})_x$ with positive eigenvalues, and analogously $E^-_q(r)_x$ for the negative ones. Since $L_{rq}\in\gau{E}$, then its eigenspaces vary continuously and $E=E^+_q(r)\oplus E^-_q(r)$ is clearly a splitting. Moreover by construction $q |_{E^+_q(r)}$ is positive definite and $q |_{E^-_q(r)}$ is negative definite. Conversely, given a $q$-splitting $E=E^+_q\oplus E^-_q$, the choice $r=q^+\oplus -q^-\in\qu{n,0}{E}$ with $q^+=q |_{E^+_q}$ and $q^-=q |_{E^-_q}$ defines a positive definite quadratic form which produces the $q$-splitting $(E^+_q,E^-_q)$ by the above procedure. Now assume we have $q$-splittings $(E^+_q(r_0), E^-_q(r_0))$ and $(E^+_q(r_1), E^-_q(r_1))$ for some $r_0,r_1\in\qu{n,0}{E}$ and consider the quadratic form on $E\times[0,1]\to X\times[0,1]$ defined by $(v,t)\mapsto q(v)$, which we still call $q$ for lack of a better name. By Proposition \ref{prop:convex} the path $t\mapsto tr_1+(1-t)r_0$ defines a positive definite quadratic form $r$ on $E\times [0,1]$ which gives a $q$-splitting of $E\times [0,1]$ by the above procedure. By restricting this splitting to $E\times\{i\}$ we obtain precisely $(E^+_q(r_i), E^-_q(r_i))$, $i=0,1$, which are then homotopic. By Remark \ref{rmk:splittings} they are also isomorphic.
\end{proof}

By the above proposition we recover the well-known fact that $E$ admits a quadratic form of inertia $(n_+,n_-)$ if and only if it admits a splitting into two rank $n_+$ and rank $n_-$ subbundles.

For each $r\in\qu{n,0}{E}$ let us denote by $\theta_r\colon\qu{\nd}{E}\to \splitt{E}$ the continuous map $q\mapsto (E^+_q(r),E^-_q(r))$. Proposition \ref{prop:splittings} implies that any $r\in\qu{n,0}{E}$ induces maps
\begin{equation}
\begin{aligned}
\theta_\sim\colon\qu{\nd}{E}&\to \splitt{E}/\sim\\
\theta_\simeq\colon\qu{\nd}{E}&\to \splitt{E}/\simeq
\end{aligned}
\end{equation}
given by assigning to each $q\in\qu{\nd}{E}$ the homotopy or isomorphism class of any of its $q$-splittings. It is the purpose of the next theorem to establish the invariance properties of these maps under the action of $\gau{E}$ and $\gaunot{E}$ on $\qu{\nd}{E}$, respectively.

First we need the following lemma.

\begin{lem}\label{lem:commonsplitting}
Given $q_0,q_1\in\qu{\nd}{E}$ and a splitting of $E$ that is both a $q_0$-splitting and a $q_1$-splitting, there is $\varphi\in\gaunot{E}$ such that $q_0=\varphi q_1$.
\end{lem}
\begin{proof}
First let us assume $q_0,q_1\in\qu{n,0}{E}$, so there is the unique common splitting $(E,0)$. Consider the path $t\mapsto q_t=tq_1+(1-t)q_0$ in $\qu{n,0}{E}$ from $q_0$ to $q_1$ and construct $L_{q_tq_0}\in\gau{E}$ as in the above discussion. We show that $\varphi=\sqrt{L}_{q_1q_0}$ satisfies the statement. Indeed $L_{q_0 q_0}=\id_{E}$ and the assignment $t\mapsto L_{q_tq_0}$ is continuous, so $L_{q_1q_0}\in\gaunot{E}$ and analogously $\sqrt{L}_{q_1q_0}\in\gaunot{E}$. Finally, by construction $q_0(u)=q_1(\sqrt{L}_{q_1q_0}u)$. Now let us consider the general case of $q_0,q_1\in\qu{\nd}{E}$. Call $E=E^+\oplus E^-$ the common splitting so $q_i=q^+_i\oplus q^-_i$, $i=0,1$ where $q^+_i=q_i |_{E^+}$ and $q^-_i=q_i |_{E^-}$. The positive definite case just proved applied to $q^+_0,q^+_1\in\qu{n_+,0}{E^+}$ and $-q^-_0,-q^-_1\in\qu{n_-,0}{E^-}$, where $n_+$ and $n_-$ are the ranks of $E^+$ and $E^-$ respectively, gives $\varphi^+\in\gaunot{E^+}$ and $\varphi^-\in\gaunot{E^-}$ such that $q^+_0=\varphi q^+_1$ and $q^-_0=\varphi q^-_1$. Then $\varphi=\varphi^+\oplus \varphi^-\in\gaunot{E}$ is the sought-for automorphism.
\end{proof}

The above lemma is a more general version of the well-known result that any two Euclidean metrics on the same vector bundle are isometric (e.g. \cite{GreHalVan72} Proposition VI p.68 or \cite{MilSta74} Problem 2-E p.24). In particular the fact that the isometry can be taken to lie in the path component of the automorphism group of the bundle seems left unreported by all published treatments so far, to the best of our knowledge.

\begin{thm}\label{thm:invariants}
The map $\theta_\simeq$ (respectively, $\theta_\sim$) is a complete invariant for the action of $\gau{E}$ (respectively, $\gaunot{E}$) on $\qu{\nd}{E}$.
\end{thm}
\begin{proof}
We are going to show that two quadratic forms $q_0, q_1$ are in the same $\gau{E}$- or $\gaunot{E}$-orbit if and only if $\theta_\simeq(q_0)=\theta_\simeq(q_1)$ or $\theta_\sim(q_0)=\theta_\sim(q_1)$, respectively. Assume $\theta_\simeq(q_0)=\theta_\simeq(q_1)$ or $\theta_\sim(q_0)=\theta_\sim(q_1)$, according to the case. By Proposition \ref{prop:key} there are $q_i$-splittings $(E^+_i, E^-_i)$, $i=0,1$, such that $f(E^+_0)=E^+_1$ and $f(E^-_0)=E^-_1$ for some $f$ lying in $\gaunot{E}$ or in $\gau{E}$, according to the case. Then $f q_1$ and $q_0$ have the splitting $(E^+_0, E^-_0)$ in common and by Lemma \ref{lem:commonsplitting} there is $g\in\gaunot{E}$ such that $(g\circ f) (q_1)=q_0$. Clearly $g\circ f$ lies in $\gaunot{E}$ or in $\gau{E}$ accordingly. We now prove the other direction for the two cases separately, starting with $\theta_\simeq$. Let us then be given two quadratic forms $q_0,q_1\in\qu{\nd}{E}$ and assume there is $f\in\gau{E}$ such that $f q_0=q_1$. If $(E^+, E^-)$ is a $q_0$-splitting, then $\left(f(E^+), f(E^-)\right)$ is a $q_1$-splitting obviously lying in the same isomorphism class of $(E^+, E^-)$. This concludes the proof for $\theta_\simeq$. Now assume $q_0,q_1\in\qu{\nd}{E}$ admit some $f\in\gaunot{E}$ such that $f q_0=q_1$ and choose a path $t\mapsto f_t\in\gaunot{E}$ such that $f_0=\id_{E}$ and $f_1=f$. This path defines a non-degenerate quadratic form (abusively denoted by) $f_t q_0$ on $E\times [0,1]\to X\times [0,1]$ which coincides with $q_i$ on $E\times \{i\}$ for $i=0,1$. Choose a $f_t q_0$-splitting $(S^+,S^-)$. Obviously $(S^+ |_{X\times\{i\}},S^- |_{X\times\{i\}})$ are $q_i$-splittings lying in the same homotopy class, and the proof is complete. 
\end{proof}

\begin{cor}
The $\gaunot{E}$-orbits of $\qu{\nd}{E}$ are precisely its path components, and the $\gau{E}$-orbits are disjoint unions of path components.
\end{cor}

There is a more geometric interpretation of $\theta_\simeq$. We recall that with each non-degenerate quadratic form $q\in\qu{n_+,n_-}{E}$ there is associated the principal $\OO{n_+,n_-}$-bundle $\fr{q}{E}\to X$ of $q$-orthonormal frames of $E$.

\begin{prop}
For any two quadratic forms $q_0,q_1\in\qu{\nd}{E}$ we have $\theta_\simeq(q_0)=\theta_\simeq(q_1)$ if and only if $\fr{q_0}{E}\simeq\fr{q_1}{E}$.
\end{prop}
\begin{proof}
The groupoid of principal $\OO{n_+,n_-}$-bundles over $X$ is equivalent to the groupoid of rank $n_++n_-$ vector bundles over $X$ equipped with a quadratic form of inertia $(n_+,n_-)$, where the morphisms are given by isometric isomorphisms. The statement then amounts to saying that $\theta_\simeq(q_0)=\theta_\simeq(q_1)$ if and only if there is $f\in\gau{E}$ such that $fq_0=q_1$, which is precisely Theorem \ref{thm:invariants}.
\end{proof}

There is another noteworthy description of $\theta_\sim$ as well. First of all we notice that the definition \ref{defn:splittings} of homotopy class of a splitting can be formulated for a single subbundle. We then have the following result.
\begin{prop}
The $\sim$-class of a splitting of $E$ is determined by the homotopy class of any of its two summands.
\end{prop}
\begin{proof}
We are going to show that for two homotopic subbundles $E_0,E_1\subset E$, any two splittings $(E_i,E_i')$ are homotopic. Indeed, let $S$ be a subbundle of $E\times[0,1]$ with $S |_{E\times\{i\}}=E_i$, $i=0,1$, and let $r_i\in\qu{n,0}{E}$ be such that $E_i'$ is the $r_i$-orthogonal complement of $E_i$. Then the $r$-orthogonal complement of $S$, where $r\in\qu{n,0}{E\times[0,1]}$ is defined by $t\mapsto r_t=tr_1+(1-t)r_0$, establishes a homotopy between the splittings $(E_i,E_i')$.
\end{proof}

Therefore $q_0$ and $q_1$ are in the same $\gaunot{E}$-orbit if and only if one (hence all) of the maximal positive-definite subbundles for $q_0$ are homotopic to one (hence all) of the maximal positive-definite subbundles for $q_1$. This is to be compared with Steenrod's approach (see \cite{Ste51} \S 40). Indeed, if we forget the second summand in a $q$-splitting then we can associate to each $q$ its maximal positive-definite subbundle $E^+_q(r)$, and this assignment determines a homotopy equivalence from $\qu{n_+,n_-}{E}$ to the space of rank $n_+$ vector subbundles of $E$. The choice of $r\in\qu{n,0}{E}$ is implicit in Steenrod.

By Theorem \ref{thm:invariants}, the existence of the $\gaunot{E}$- and $\gau{E}$-actions brings some order in the space $\qu{\nd}{E}$. We have already seen that its path components are precisely the $\gaunot{E}$-orbits. Now we prove the following (see Lemma 4.1 in \cite{Mou01}).

\begin{cor}
Each two path components of $\qu{\nd}{E}$ lying in the same $\gau{E}$-orbit are homeomorphic.
\end{cor}
\begin{proof}
Assume $fq_0=q_1$ for $q_i\in\qu{\nd}{E}$, $i=0,1$, $f\in\gau{E}$. Since $f$ acts as a homeomorphism on $\qu{\nd}{E}$ and homeomorphisms send path components to path components bijectively, we deduce that restricting the homeomorphism induced by $f$ to the path component of $q_0$ gives a homeomorphism onto the path component of $q_1$.
\end{proof}

Another easy consequence of Theorem \ref{thm:invariants} is the following.

\begin{cor}
The cardinality of the set of path components in a given $\gau{E}$-orbit is bounded by the cardinality of the group $\gau{E}/\gaunot{E}$.
\end{cor}

The actual cardinality of the set of path components in a given $\gau{E}$-orbit might be investigated in principle by studying the stabiliser of a quadratic form, that is its isometry subgroup in $\gau{E}$. The lack of local compactness of the groups involved does not permit to use the full power of topological group theory though, making the task seemingly more challenging.

\section{Acknowledgments}
I warmly thank Gharchia Abdellaoui, Sergio A. H. Cardona, Pietro Giavedoni, and Antonio Moro for useful discussions and for giving me precious good mood during the making of this paper.

\end{document}